\documentclass[11pt,oneside]{amsart}

\usepackage{amsmath,amsthm,amssymb,amsfonts,amscd}
\usepackage{url}
\usepackage{color}
\usepackage{setspace}
\usepackage{enumerate}

\usepackage{epsfig}
\usepackage{graphicx}
\usepackage{wrapfig}
\usepackage{mathrsfs}
\usepackage{pinlabel}

\usepackage{footnote} % \usepackage{smfenum} % paquet introuvable
\usepackage{mathrsfs}

\newcounter{notes}%

\usepackage[T1]{fontenc}
\usepackage{lmodern}
\theoremstyle{plain}
\newtheorem{theorem}{Theorem}
\newtheorem*{theoremA}{Theorem A}
\newtheorem*{theorem*}{Theorem}

\newtheorem{lemma}[theorem]{Lemma}

\theoremstyle{definition}

\newtheorem{conjecture}[theorem]{Conjecture}

\usepackage{graphicx} % Allows including images
\usepackage{epsfig}
\usepackage{pinlabel}
\usepackage{mathabx}
\usepackage{tikz-cd} % CAUSE LE PLANTAGE (IL Y A 3 DIAGRAMMES COMMUTATIFS, CACHES DERRIERE DES %%%%). 
\usepackage{booktabs} % Allows the use of \toprule, \midrule and \bottomrule in tables

\title{Uniform Lipschitz extension in bounded curvature}           % ``d'apr\`es ...'' est un sous-titre
\author{Fran\c{c}ois GU\'ERITAUD}
\address{CNRS \& Universit\'e de Lille \\ Laboratoire Paul Painlev\'e \\ 59655 Villeneuve d'Ascq Cedex, France}
\email{Francois.Gueritaud@univ-lille.fr}
\thanks{This work was partially supported by the Agence Nationale de la Recherche under the grant DiscGroup (ANR-11-BS01-013) and through the Labex \textsc{Cempi} (ANR-11-LABX-0007-01).
It was partially completed during a residence term at \textsc{Ih\'es}, where it received funding from the European Research Council under the EU Horizon 2020 research and innovation program (ERC starting grant DiGGeS, grant agreement 715982).
The FlixBus company provided opportunities to devote time and focus to this material. \\
\emph{Keywords: negative curvature, Lipschitz extension. MSC-2010: 54C20.}
}

\begin{document}

\begin{abstract} We prove a uniform extension result for contracting maps defined on subsets of Hadamard manifolds subject to curvature bounds.\end{abstract}

\maketitle

\section*{Introduction}
\subsection*{Lipschitz extension problem}
Let $X$, $Y$ be metric spaces.
Consider $X'\subset X$ and a Lipschitz map $f:X'\rightarrow Y$.
Can we extend $f$ to $F:X\rightarrow Y$ with the same constant $\mathrm{Lip}(F)=\mathrm{Lip}(f)$? 
Failing that, can we bound the loss?
This potential ``loss'' can be encapsulated in a function $\mathcal{L}_{X,Y}$:
$$\begin{array}{rlcl} \mathcal{L}_{X,Y}: & \mathbb{R}^+ & \longrightarrow & \quad \mathbb{R}^+ \\ & C & \longmapsto &\displaystyle  
\sup_{\substack{X'\subset X \\ f:X'\rightarrow Y \\ \mathrm{Lip}(f)\leq C}} \; 
%\sup_{\substack{X\supset X' \underset{f}\rightarrow Y \\ \mathrm{Lip}(f)\leq C}} \; 
\inf_{\substack{F:X\rightarrow Y \\ F|_{X'}=f}} \; 
\mathrm{Lip}(F).\end{array}$$
For example, maps to~$\mathbb{R}$, or more generally to a metric tree $T$, can always be extended without loss~\cite{m, las}: $\mathcal{L}_{X, \mathbb{R}}(C)=\mathcal{L}_{X, T}(C)=C$ for all $C\geq 0$.
Kirszbraun~\cite{k} proved that $\mathcal{L}_{X, Y}(C)=C$ when $X,Y$ are Euclidean spaces.

Recall that a Hadamard manifold is a complete, simply connected Riemannian manifold of nonpositive sectional curvature.
Lang and Schr\"oder~\cite{las}, extending work of Valentine~\cite{v} for the constant-curvature case, proved:
\begin{theoremA} \cite{las}
Let $\kappa_0, \kappa'_0<0$ be constants.
If $X,Y$ are Hadamard manifolds with $\kappa_X \geq \kappa_0$ and $\kappa_Y \leq \kappa'_0$, then $\mathcal{L}_{X,Y}(C)=C$ for all $C\geq \sqrt{\kappa_0/\kappa'_0}$.
\label{lasch}
\end{theoremA}
\subsubsection*{Main result}
Up to scaling, we may and always will assume $\kappa_0=\kappa'_0=-1$.
In that case, the above theorem also gives: $\mathcal{L}_{X,Y}(C)\leq 1$ when $C\leq 1$. 
The goal of this note is to prove the following refinement:
\begin{theorem} \label{main}
For any $C<1$, $K\leq -1$ and $m \in \mathbb{N}$, there exists $C'<1$ such that 
for any Hadamard manifolds $X,Y$ of dimension $\leq m$ 
satisfying $\kappa_X\geq -1 \geq \kappa_Y \geq K$, 
one has $\mathcal{L}_{X,Y}(C)\leq C'$.
\end{theorem}
The $X=Y=\mathbb{H}^2$ case was conjectured in~\cite[App.\ C]{guk}, which put forward a strategy when $X'$ has bounded diameter.

\subsection*{About the method}
Our proof is based on the template of Lang and Schr\"oder's proof of Theorem~A, which we will recall in~\S\ref{prolasch} (slightly simplified, as~\cite{las} is set in the context of Alexandrov spaces).
The extra ingredients, which extend and uniformize arguments of~\cite{guk}, are based on the notion that under negative curvature, both in the small-scale limit (Euclidean geometry) and large-scale limit (real trees), loss-less extension is known to hold.
Thus, loss ($\mathcal{L}_{X,Y}(C)>C$) is in a sense a medium-range phenomenon, and can be controlled using a form of compactness and covering arguments.

When extending $f:X'\rightarrow Y$ to a single point $\xi \in X\smallsetminus X'$, we will see in~\S\ref{prolasch} that there is usually a natural ``optimal'' image $F(\xi)$, relative to the set $X'$ where the map is already defined. 
Given a second point $\xi'$, we can then assign it an optimal image relative to $X'\cup \{\xi\}$, then pass to a third point $\xi''$ and so on, studying the loss incurred at each step.
One difficulty, which could cause the losses to pile up, is that the notion of ``optimal'', being relative to $X'\cup\{\xi, \xi', \dots\}$, changes as we go.

However, as pointed out in~\cite{las}, this difficulty \emph{disappears} when $Y$ is a metric tree: then, taking each $\xi\in X\smallsetminus X'$ to its optimal image (relative to $X'$ only) yields a globally Lipschitz map, with no loss.
This key feature, together with the fact that the curvature bounds force $Y$ coarsely to behave somewhat like a tree at large distances, is what allows us to prove Theorem~\ref{main}.
To patch together maps defined on different regions of $X$, we will use a standard interpolation procedure described in~\S\ref{interpol}.

\subsubsection*{Plan}
Section~\S\ref{prolasch} recalls the proof of Theorem~A; Section~\S\ref{secmain} proves Theorem~\ref{main}. Section~\S\ref{seconj} indulges in some speculation.

\subsection*{Notation}
Distances in metric spaces are all denoted $d$.

\noindent The open ball centered at $\xi$, of radius $r$, is written $\mathbb{B}_\xi(r)$. For a ball of unspecified center, we sometimes write $\mathbb{B}(r)$.

\noindent Given a point $\xi$ in a Hadamard manifold $X$, we write $\exp_\xi: \mathrm{T}_\xi(X) \rightarrow X$ the exponential map, and $\log_\xi$ its inverse.

\noindent Given $x,z\in X\smallsetminus \{\xi\}$, the notation $\widehat{x\xi z} \in [0,\pi]$ then refers to the angle between vectors $\log_\xi(x)$ and $\log_\xi(z)$, for the Euclidean metric on $\mathrm{T}_\xi(X)$.

\noindent The volume measure on $X$ is written $\mathrm{Vol}_X$.

\subsection*{Acknowledgements}
My interest for the Lipschitz extension problem in negative curvature started with Fanny Kassel and our work~\cite{guk}, which tackled an equivariant version. I am grateful to her and to Urs Lang for discussions related to this material.

\section{Proof of Theorem~A}\label{prolasch}
To build loss-less extensions, it is enough to do it one point $\xi\in X$ at a time: indeed, we can then repeat for a dense sequence $(\xi_n)_{n\in \mathbb{N}}$ of $X$, and pass to all of $X$ by continuity. 

Let $X'\subset X$ and $f:X'\rightarrow Y$ be $C$-Lipschitz with 
$$C\geq 1,$$
where $X,Y$ are Hadamard manifolds subject to curvature bounds $\kappa_X \geq -1 \geq \kappa_Y$.
We can restrict attention to $X'$ compact, nonempty. Consider $\xi \in X\smallsetminus X'$: the function defined by
$$\begin{array}{rcl} \varphi_\xi:  Y & \longrightarrow & \mathbb{R}^+  \\ 
    y  & \longmapsto & \displaystyle \max_{x\in X'} \, \frac{d(y,f(x))}{d(\xi, x)}
\end{array}$$
is proper and convex on $Y$, hence achieves a minimum 
\begin{equation}\label{cxi} C_\xi:=\min \varphi_\xi = \varphi_\xi (\eta) \end{equation}
at some $\eta\in Y$ (in fact unique). 
We can think of $\eta$ as an ``optimal candidate for $F(\xi)$'': Theorem~A will follow if we can prove 
$$C_\xi \leq C.$$ 
If $C_\xi \leq 1$ we are done. If $C_\xi\geq 1$, define the compact set 
\begin{equation}\label{xxi} X_\xi:=\Big \{x\in X', \frac{d(\eta, f(x))}{d(\xi, x)}=C_\xi \Big \}. \end{equation}

The exponential of any linear hyperplane $V \subset \mathrm{T}_\eta Y$ separates $Y$ into two half-spaces, each of which contains points of $f(X_\xi)$ in its closure: if not, we could push $\eta$ towards $f(X_\xi)$ (perpendicularly to $V$) to reduce $\varphi_\xi(\eta)$, contradicting minimality. 
Hence, $\eta$ belongs to the convex hull of some points $y_i=f(x_i)$, $1\leq i \leq n$ where $x_i \in X_\xi$:
\begin{equation} \label{convhull} \sum_{i=0}^n \lambda_i \log_\eta (y_i) = 0_\eta \in \mathrm{T}_\eta Y \end{equation}
for some reals $\lambda_i>0$.
Since $x_i\in X_\xi$, the lengths $\ell_i:=d(\xi,x_i)= \Vert \log_\xi (x_i) \Vert$ satisfy $C_\xi \ell_i = d(\eta, y_i)= \Vert \log_\eta (y_i) \Vert$. 
We can then write
\begin{align*} 0 &\leq \Big \Vert C_\xi \sum_{i=0}^n \lambda_i \log_\xi (x_i) \Big \Vert^2 - \Big \Vert \sum_{i=0}^n \lambda_i \log_\eta (y_i) \Big \Vert^2 \\
&= C_\xi^2 \sum_{i,j} \lambda_i \lambda_j \,  \ell_i \ell_j \, \big( \cos \widehat{x_i \xi x_j} - \cos \widehat{y_i \eta y_j} \big )
\end{align*}
%$$ 0 \leq \Big \Vert C_\xi \sum_{i=0}^n \frac{\log_\xi (x_i)}{\lambda_i} \Big \Vert^2 - \Big \Vert \sum_{i=0}^n \frac{\log_\eta (y_i)}{\lambda_i} \Big \Vert^2 = C_\xi^2 \sum_{i,j} \ell_i \ell_j \frac{\cos \widehat{x_i \xi x_j} - \cos \widehat{y_i \eta y_j}}{\lambda_i \lambda_j} $$
hence at least one summand with $i\neq j$ is $\geq 0$, which happens if and only if $\widehat{x_i \xi x_j} \leq \widehat{y_i \eta y_j}$. 
Hence, up to reindexing, we may assume 
\begin{equation} \theta := \widehat{x_1 \xi x_2} \leq \widehat{y_1\eta y_2} =:\theta'. \label{cot} \end{equation}

Let $\mathcal{D}_\theta(\ell, \ell')$ denote the distance, in the hyperbolic plane $\mathbb{H}^2$, between the far ends of two segments of lengths $\ell, \ell'$ starting from a common vertex, an angle $\theta$ apart.
A~well-known trigonometric formula gives explicitly
\begin{equation} \label{trigo}
\mathcal{D}_\theta(\ell, \ell') = \mathrm{Arccosh} (\cosh \ell \, \cosh \ell' - \sinh \ell \, \sinh \ell' \, \cos \theta) 
\end{equation}
but we will mostly use the following facts: the function $\mathcal{D}_\theta$ is convex in its two arguments, vanishes at $(0,0)$, and depends monotonically on $\theta$.
The Cartan-Alexandrov-Toponogov or CAT($-1$) comparison inequalities, whose interesting history is recounted in~\cite{cat}, say that
% \begin{equation}
% d(x_1, x_2)\leq \mathcal{D}_{\theta}(\ell_1, \ell_2)~\text{ and }~d(y_1, y_2)\geq \mathcal{D}_{\theta'}(C_\xi \ell_1, C_\xi \ell_2). \label{CAT}
% \end{equation}
\begin{equation} \label{CAT}
\begin{array}{llll}
\mathrm{a.} & d(x_1, x_2) & \leq & \mathcal{D}_{\theta}(\ell_1, \ell_2) \\ 
\mathrm{b.} & d(y_1, y_2) & \geq & \mathcal{D}_{\theta'}(C_\xi \ell_1, C_\xi \ell_2) 
\end{array} \end{equation}
due to the curvature bounds $\kappa_X \geq -1 \geq \kappa_Y$.
Therefore,

\begin{align} 
 C\, d(x_1, x_2) & \geq  d(y_1, y_2) & \text{(Lipschitz bound)} \notag \\
& \geq \mathcal{D}_{\theta'}(C_\xi\ell_1, C_\xi\ell_2) & \text{by \eqref{CAT}}&\mathrm{.b} \notag \\
& \geq \mathcal{D}_{\theta}(C_\xi\ell_1, C_\xi\ell_2) & \text{by \eqref{cot}} \notag \\ 
& \geq  C_\xi \, \mathcal{D}_{\theta}(\ell_1, \ell_2) & \label{conv} \\
& \geq  C_\xi \, d(x_1, x_2). & \text{by \eqref{CAT}}&\mathrm{.a} \notag 
\end{align}
where \eqref{conv} uses convexity of $\mathcal{D}_\theta$ and $C_\xi\geq 1$. Hence $C_\xi\leq C$ as desired, 
%: extending $f$ by $\xi\mapsto \eta$ is loss-less if $\mathrm{Lip}(f)=C\geq 1$. 
proving Theorem~A. \qed

\section{Proof of Theorem~\ref{main}} \label{secmain}

% This proof extends to show that tree-valued Lipschitz maps admit loss-less extensions.
% In $\mathbb{S}^2$, the function $\mathcal{D}_\theta:[0,\pi]^2\rightarrow \mathbb{R}^+$ becomes radially \emph{concave}: hence, loss-less extension rather holds for $\mathrm{Lip}(f)=C\leq 1$.
\subsection{One-point extension}
We start by bounding the loss for extensions to a single point.
% Observation:
% The blind extension $F_*:X\rightarrow Y$ is loss-less if $Y$ is a tree.
% 
% Proof: suppose $F_*$ takes $\xi, \xi' \in X\smallsetminus X'$ to $\eta, \eta'$ such that
% $$ \begin{array}{cccc}
% d(\eta, y_i)  =C_\xi d(\xi, x_i)  & \text{ where } f(x_i)=y_i  & (i=1,2) & \text{ and } \eta\in [y_1, y_2] \\ 
% d(\eta', y'_i)  =C_{\xi'} d(\xi', x'_i)  & \text{ where } f(x'_i)=y'_i & (i=1,2)& \text{ and } \eta'\in [y'_1, y'_2].
% \end{array}$$
% We may assume $C_{\xi'}\leq C_\xi$ and $\eta\in[\eta',y_1]$. Then
% $$ 
% \begin{array}{rl}
%  d(\eta', \eta) & = d(\eta', y_1)-d(y_1, \eta) \\ 
%  & \leq C_{\xi'}\,  d(\xi',x_1)-C_\xi d(x_1, \xi) \\
%  & \leq C_\xi \, (d(\xi', x_1)-d(x_1, \xi)) \\
%  & \leq C_\xi \, d(\xi',\xi).
%  \end{array}
% $$

\begin{lemma} \label{1point}
For any $C<1$ there exists $C^*<1$ such that for any Hadamard manifolds $X,Y$ satisfying $\kappa_X \geq -1 \geq \kappa_Y$, any $X'\subset X$ and any $\xi\in X\smallsetminus X'$, every $C$-Lipschitz map $f:X'\rightarrow Y$ has a $C^*$-Lipschitz extension to $X'\sqcup \{\xi\}$.
\end{lemma}
\begin{proof}
Take $f, C, \xi$ as in the statement and define $C_\xi\geq 0$ (as well as $\eta \in Y$, $X_\xi \subset X'$, $y_i=f(x_i)\in f(X_\xi)$ and $\ell_i=d(\xi, x_i)$) as in the previous proof. 
Theorem~A gives $C_\xi\leq 1$: let us bound $C_\xi$ away from $1$ in terms of $C$ alone.

Let $\Delta>0$ be such that
\begin{equation} \label{addloss} \mathcal{D}_{\pi/2}(\ell, \ell') \geq \ell + \ell' - \Delta ~\text{ for all }~ \ell, \ell'\geq 0 \end{equation}
(using~\eqref{trigo} one can show $\Delta=\log 2$ works).
Let $r>0$ be large enough that 
\begin{equation} \label{C'} \widehat C:=C+{\Delta}/{r}<1. \end{equation}
We distinguish two cases.

$\bullet$ If $\ell_i \geq r$ for some index $i$, we use~\eqref{convhull} to find $j\neq i$ such that 
\begin{equation} \label{obtus} \theta':= \widehat{y_i \eta y_j} \geq \pi/2 \end{equation} 
and write:
\begin{align} 
C\, d(x_i, x_j) & \geq  d(y_i, y_j) & \text{(Lipschitz bound)} \notag \\
& \geq \mathcal{D}_{\theta'}(C_\xi\ell_i, C_\xi\ell_j) & \text{by \eqref{CAT}}&\mathrm{.b} \notag \\
& \geq C_\xi (\ell_i + \ell_j) - \Delta & \text{by \eqref{addloss}--\eqref{obtus}} \notag \end{align}
hence $$C_\xi\leq %C+\frac{\Delta}{r} = 
\widehat C$$ 
by \eqref{C'}, due to the triangle inequality $d(x_i, x_j) \leq \ell_i+\ell_j$ and $\ell_i \geq r$.

$\bullet$ If no such index $i$ exists, then we define $x_1, x_2\in X_\xi$ and $\theta\leq \theta' \in [0,\pi]$ as in the proof of Theorem~A and write, similar to~\eqref{conv}:
\begin{align}
C\, d(x_1, x_2) & \geq  d(y_1, y_2) & \text{(Lipschitz bound)} \notag \\
& \geq \mathcal{D}_{\theta'}(C_\xi\ell_1, C_\xi\ell_2) & \text{by \eqref{CAT}}&\mathrm{.b} \notag \\
& \geq \mathcal{D}_{\theta}(C_\xi\ell_1, C_\xi\ell_2) & \text{by \eqref{cot}} \notag \\
& \geq C'_\xi \, \mathcal{D}_{\theta}(\ell_1, \ell_2) & \text{(\emph{see \eqref{newc} below})}\label{pattes} \\ 
& \geq C'_\xi \, d(x_1, x_2) & \text{by \eqref{CAT}}&\mathrm{.a} \notag
\end{align}
where we use the new constant 
\begin{equation} \label{newc} C'_\xi:= \frac{\sinh (C_\xi r)}{ \sinh (r)}. \end{equation}
Indeed, for a basepoint $o\in \mathbb{H}^2$, the differential of the exponential map $\mathrm{exp}_o:(\mathbb{R}^2,0) \rightarrow (\mathbb{H}^2,o)$ at 
a point of the circle $\partial \mathbb{B}_0(\lambda)$
%a distance $\lambda$ from the origin 
has principal values~$1$ radially and $\sinh (\lambda)$ along the circle --- this can be checked by differentiating~\eqref{trigo} near $(\ell, \ell', \theta)=(\lambda, \lambda, 0)$. 
It follows that the radial map 
$$ \begin{array}{rrcl} H: & (\mathbb{H}^2,o) & \longrightarrow & (\mathbb{H}^2,o) \\
& x & \longmapsto & \exp_{o} \big (C_\xi \log_o (x) \big ), \end{array} $$
defining a homothety of ratio $C_\xi$ on each line through $o$, satisfies
$$\mathrm{Lip}\left ( \left . H^{-1} \right |_{\mathbb{B}_o(C_\xi r)} \right ) = \frac{\sinh (r)}{\sinh(C_\xi r)}= \frac{1}{C'_\xi}$$
which means that step \eqref{pattes} holds (using $\ell_1, \ell_2\leq r$). 
Therefore, $C'_\xi \leq C$. 
Substituting in~\eqref{newc}, we find
\begin{equation} \label{arcsinh} C_\xi \leq \frac{\mathrm{Arcsinh}(C\sinh (r))}{r}<1. \end{equation}

In either case, we have bounded the Lipschitz constant $\mathrm{max}\{C,C_\xi\}$ (for the one-point extension $\xi \mapsto \eta$ of $f$) uniformly away from $1$.
\end{proof}

\subsection{Averaging maps} \label{interpol}
In curvature $\leq 0$, convex interpolation behaves well with respect to Lipschitz constants.
Namely, given $f_0, f_1 : X \rightarrow Y$, let $(f_t(x))_{t\in [0,1]}$ be the constant-speed parametrization of the geodesic segment $[f_0(x), f_1(x)]$, for all $x\in X$.
The ``barycenter'' maps $f_t:X\rightarrow Y$ thereby defined satisfy:
\emph{if $f_1$ agrees with $f_0$ on $X' \subset X$ then so does $f_t$.}
%$\{f_0=f_1\}\subset \{f_0=f_t\}$, hence .

Moreover, for all $x,x'\in X$, if $(y_t)_{t\in [0,1]}$ denotes the constant-speed para\-metrization of the segment $[f_0 (x), f_1(x')]$, then
\begin{align*}
 d(f_t(x),f_t(x')) & \leq d(f_t(x),y_t) + d(y_t,f_t(x')) \\
 & \leq t\,  d(f_1(x),f_1(x')) + (1-t)\, d(f_0(x),f_0(x'))
\end{align*}
by CAT(0) comparison inequalities. It follows that 
$$ \mathrm{Lip}(f_t) \leq t\, \mathrm{Lip}(f_1) + (1-t)\, \mathrm{Lip}(f_0).$$
We will simply use the \emph{notation} $$f_t =: t\, f_1 + (1-t)\, f_0.$$
%the right member being a \emph{notation}. 
We can also iterate the construction above, to define barycenters of $N$ maps: given maps $(f_i)_{i\geq 1}$, the maps $F_N=\sum_{i=1}^N \frac{1}{N} f_i$, defined inductively on $N$ by 
% $$ F_N : = \frac{1}{N} f_N + \frac{N-1}{N} \Big ( \sum_{i=1}^{N-1} \frac{1}{N-1} f_i \Big ) $$
$F_N : = \frac{1}{N} f_N + \frac{N-1}{N} \big ( \sum_{i=1}^{N-1} \frac{1}{N-1} f_i \big )$, inductively satisfy for any $Z\subset X$:
\begin{equation} \label{moyenne} 
\mathrm{Lip}(F_N|_Z) \leq \sum_{i=1}^N \frac{1}{N} \mathrm{Lip}(f_i|_Z).
\end{equation}
When $N\geq 3$ this construction is not robust under permutation of the $f_i$; note however that symmetric constructions do exist~\cite{las}, which also satisfy a weakened form of associativity~\cite{guk}.

\subsection{Extensions to the whole space}
We now prove Theorem~\ref{main}. 
Let $C<1$, $K\leq -1$, $m\in \mathbb{N}$ and Hadamard manifolds $X,Y$ be as in the theorem, and $C^*\in[C,1)$ be given by Lemma~\ref{1point}.

Let $f:X'\rightarrow Y$ be a $C$-Lipschitz map, where $X'\subset X$.
Again, we may assume $X'$ is compact.
By Lemma~\ref{1point}, we may consider a family of $C^*$-Lipschitz extensions $(f_\xi^*)_{\xi\in X}$ to $X'\cup\{\xi\}$, taking $\xi$ to its optimal candidate image. 
We do allow $\xi \in X'$, in which case $f_\xi^*=f$.
Small balls in $X$ and $Y$ are uniformly $(1+o(1))$-bi-Lipschitz to Euclidean balls, by the curvature bounds $0\geq \kappa_X, \kappa_Y \geq K$ 
(in fact CAT-type inequalities~\eqref{CAT} show that this $o(1)$ tolerance is quadratic in the size of the balls). 
By composition, loss-less extension in Euclidean geometry~\cite{k} implies that there exists $\varepsilon_0 \in (0,1)$ such that each $f_\xi^* \big |_{\mathbb{B}_\xi(\varepsilon_0) \cap (X'\cup \{\xi\})}$
has a $\sqrt{C^*}$-Lipschitz extension 
\begin{equation} \label{patch} \widehat{f_\xi} : \mathbb{B}_\xi(\varepsilon_0) \longrightarrow Y. \end{equation}
Let $\varepsilon < \varepsilon_0$ be small enough, and $R>1$ large enough, that 
% \begin{equation} \label{buffer} \frac{C^* +  \varepsilon/\varepsilon_0}{1-  \varepsilon/\varepsilon_0}\leq 1 \end{equation}
% \begin{equation} \label{buffer2} \frac{(C^*+{\Delta}/{R})+2\varepsilon/R}{1-2\varepsilon/R} \leq 1.\end{equation}
\begin{equation} \label{buffer} \mathrm{(a)} \quad  \frac{C^* +  \varepsilon/\varepsilon_0}{1-  \varepsilon/\varepsilon_0}\leq 1 
\quad\quad \text{and} \quad\quad \mathrm{(b)} \quad
\frac{(C^*+{\Delta}/{R})+2\varepsilon/R}{1-2\varepsilon/R} \leq 1\end{equation}
where $\Delta>0$ still satisfies~\eqref{addloss}.

\begin{lemma} \label{qtree}
Let $\xi, \xi' \in X$ be distance $\geq R$ apart. 
Then $\mathrm{Lip}(G)\leq 1$ for $$G:= f \sqcup \widehat{f_\xi}\big |_{\mathbb{B}_\xi(\varepsilon)} \sqcup \widehat{f_{\xi'}} \big |_{\mathbb{B}_{\xi'}(\varepsilon)}.$$
\end{lemma}
\begin{proof}
Consider $x,x'\in X'\cup \mathbb{B}_\xi(\varepsilon) \cup \mathbb{B}_{\xi'}(\varepsilon)$. We distinguish several cases.

\noindent $\bullet$ (i) If $x,x'\in X'$ then $d(G(x), G(x'))=d(f(x), f(x'))\leq C d(x,x')$ because $f$ is $C$-Lipschitz.

\noindent $\bullet$ (ii) If $x,x'\in \mathbb{B}_\xi(\varepsilon)$ then by construction of $\widehat{f_\xi}$,
\begin{equation} \label{nest}
d(G(x), G(x'))=d(\widehat{f_\xi}(x), \widehat{f_\xi}(x'))\leq \sqrt{C^*} d(x,x'). 
\end{equation}

\noindent $\bullet$ (iii) If $x,x'\in \mathbb{B}_{\xi'}(\varepsilon)$, we do as in (ii), exchanging $\xi$ and $\xi'$. 

\noindent $\bullet$ (iv) If $x\in \mathbb{B}_\xi(\varepsilon)$ and $x'\in X'$, we distinguish two cases:
if $x'\in X'\cap \mathbb{B}_\xi(\varepsilon_0)$, then~\eqref{nest} still applies. If not, then
we compute
% \begin{align*} \frac{d(G(x'), G(x))}{d(x',x)} 
% & \leq \frac{d(G(x'), G(\xi)) + d(G(\xi), G(x))}{d(x',\xi) - d(\xi, x)} \\ 
% & \leq \frac{C^* d(x',\xi) + d(\xi, x)}{d(x',\xi) - d(\xi, x)} 
% %\\ & = \frac{C^*+{d(\xi,x)}/{d(x',\xi)}}{1-{d(\xi,x)}/{d(x',\xi)}} 
% & \leq 1 ~ \text{ by~\eqref{buffer}}\end{align*}
$$ \frac{d(G(x'), G(x))}{d(x',x)} 
\leq \frac{d(G(x'), G(\xi)) + d(G(\xi), G(x))}{d(x',\xi) - d(\xi, x)} 
 \leq \frac{C^* d(x',\xi) + d(\xi, x)}{d(x',\xi) - d(\xi, x)} $$
which is $\leq 1$ by~\eqref{buffer}.a, since $d(\xi,x)\leq \varepsilon$ and $d(x',\xi)\geq \varepsilon_0$.

\noindent $\bullet$ (v) If $x\in \mathbb{B}_{\xi'}(\varepsilon)$ and $x'\in X'$, we do as in (iv), exchanging $\xi$ and $\xi'$. 

\noindent $\bullet$ (vi) Up to exchanging $x$ and $x'$, the only remaining case is that $x\in \mathbb{B}_{\xi}(\varepsilon)$ and $x'\in \mathbb{B}_{\xi'}(\varepsilon)$.
It is only here that we will use the assumption $d(\xi, \xi')\geq R$.

We first treat the case $(x,x')=(\xi, \xi')$.
Recall from~\eqref{cxi} the optimal candidates $\eta=G(\xi)$ and $\eta'=G(\xi')$ and optimal constants
$C_\xi, C_{\xi'}<1$ used in the proofs of Theorem~A and Lemma~\ref{1point}. 
By symmetry, we may assume 
\begin{equation} \label{match3}C_{\xi'} \leq C_\xi  \end{equation}
and by definition of $C_{\xi'}$ we have 
\begin{equation} \label{match} {d(\eta',f(z))} \leq C_{\xi'}\, {d(\xi', z)} ~\text{ for all }~z\in X'.\end{equation}
Recall also from~\eqref{xxi} the compact subset $X_\xi \subset X'$, satisfying
\begin{equation} \label{match2} {d(\eta,f(z))} = C_{\xi}\, {d(\xi, z)} ~\text{ for all }~z\in X_\xi. \end{equation}
By Lemma~\ref{1point} we know
\begin{equation} \label{match4}  C_\xi \leq C^* < 1. \end{equation}
Since $\eta$ lies by~\eqref{convhull} in the convex hull of $f(X_\xi)$, we can find $y_1=f(x_1) \in f(X_\xi)$ such that $\widehat{\eta' \eta y_1} \geq \frac{\pi}{2}$. Then, 
\begin{align}
 d(\eta, \eta')  & \leq d(y_1,\eta')-d(y_1, \eta)+\Delta & \text{by \eqref{CAT}.b--\eqref{addloss}} \notag \\ 
 & \leq C_{\xi'} d(x_1,\xi')-C_\xi d(x_1, \xi) + \Delta & \text{by \eqref{match}--\eqref{match2}} \notag \\
 & \leq C_\xi \left ( d(x_1,\xi')-d(x_1, \xi) \right )+ \Delta & \text{by \eqref{match3}} \label{comput} \\
 & \leq C_\xi \, d(\xi', \xi) + \Delta & \text{(triangle inequality)} \notag \\
 & \leq C^* \, d(\xi', \xi) + \Delta & \text{by \eqref{match4}} & . \notag 
\end{align}
 Since by assumption $d(\xi, \xi')\geq R$, it follows that 
\begin{align} 
{d(\eta, \eta')}/{d(\xi, \xi')} & \leq C^*+{\Delta}/{R}  \quad <1 & \text{by \eqref{buffer}.b.}  \label{buffer3} \end{align}
This 
(\footnote{~For $Y$ a tree and $X$ a general metric space, a variant of the computation~\eqref{comput} holds with $\Delta=0$, and a variant of the argument in~\S\ref{prolasch} yields $C_\xi\leq C$.
Taking each $\xi, \xi' \in X \smallsetminus X'$ (independently) to its optimal image $\eta, \eta' \in Y$ therefore produces a global, loss-less extension of $f$: this was proved in~\cite[Th.~B]{las}, as alluded to in the Introduction.}) 
deals with the case $(x,x')=(\xi, \xi')$.

The general case of (vi) is now similar to (iv-v): we can compute
\begin{align*} \frac{d(G(x), G(x'))}{d(x,x')} 
& \leq \frac{d(G(x), \eta) + d(\eta, \eta') + d(\eta', G(x'))}{-d(x,\xi) + d(\xi, \xi') - d(\xi', x')} \\ 
& \leq \frac{(C^* + \Delta/R) d(\xi,\xi') + 2\varepsilon }{d(\xi,\xi') - 2 \varepsilon} & \text{by~\eqref{buffer3}} \\
& = \frac{(C^* + \Delta/R) +{2\varepsilon}/{d(\xi,\xi')}}{1-{2\varepsilon}/{d(\xi,\xi')}} \quad \leq 1 & ~\text{by~\eqref{buffer}}&\mathrm{.b,} \end{align*}
using again $d(\xi, \xi')\geq R$.
Therefore, $\mathrm{Lip}(G)\leq 1$.
\end{proof}

To finish proving Theorem~\ref{main}, consider a maximal $\varepsilon$-sparse subset 
$$\Xi=\big \{\xi_i \big \}_{i\in \mathbb{N}} \subset X.$$ 
This means that the closed balls $\overline{\mathbb{B}_{\xi_i}(\varepsilon)}$ cover $X$ but the $\mathbb{B}_{\xi_i}(\varepsilon/2)$ are pairwise disjoint (i.e.\ the $\xi_i \in \Xi$ are mutually $\geq \varepsilon$ apart).
For example, $\Xi$ can be constructed from a dense sequence $(x_\iota)_{\iota\in \mathbb{N}}$ of $X$ by setting $\xi_1:=x_1$ and letting inductively $\xi_i$ be the first $x_\iota$ lying outside $\mathbb{B}_{\xi_1} (\varepsilon) \cup \dots \cup \mathbb{B}_{\xi_{i-1}} (\varepsilon)$. 

Since $0\geq \kappa_X \geq -1$, the volume of a ball in $X$ is bounded above (resp.\ below) by the volume of a ball of the same radius in hyperbolic space $\mathbb{H}=\mathbb{H}^{\dim(X)}$ (resp.\ in Euclidean space $\mathbb{E}=\mathbb{R}^{\dim(X)}$): indeed, CAT-type inequalities~\eqref{CAT} show that the \emph{Jacobians} of the exponential maps in $\mathbb{H}$, $X$, and $\mathbb{E}$ form, in that order, a weakly decreasing sequence.
Let $N \in \mathbb{N}$ satisfy
\begin{equation} \label{bins}
N \geq \frac{\mathrm{Vol}_{\mathbb{H}}(\mathbb{B}(R+\varepsilon/2))}{\mathrm{Vol}_{\mathbb{E}}(\mathbb{B}(\varepsilon/2))}.
\end{equation}
Each ball $\mathbb{B}_{\xi_i}(R)$ contains at most $N$ points of $\Xi$, because the $\varepsilon/2$-balls centered at those points are disjoint and contained in $\mathbb{B}_{\xi_i}(R+\varepsilon/2)$.
Therefore, we can find a partition of $\Xi$ into ``bins''
$$\Xi=\Xi_1 \sqcup \dots \sqcup \Xi_N$$ 
such that any distinct $\xi, \xi' \in \Xi_j$ satisfy $d(\xi, \xi')\geq R$:
for example, the $\Xi_j$ can be constructed inductively by putting $\xi_1$ in $\Xi_1$, and then dropping in turn each $\xi_i$ into any bin $\Xi_j$ disjoint from $\{\xi_1, \dots, \xi_{i-1} \} \cap \mathbb{B}_{\xi_i}(R)$. 

Recall from~\eqref{patch} the $\sqrt{C^*}$-Lipschitz maps $\widehat{f_{\xi_i}}$ defined in $\varepsilon_0$-neighborhoods of the $\xi_i$.
For each $1\leq j \leq N$, define the map 
$$F_j:= \Big ( \bigsqcup_{\xi\in \Xi_j} \widehat{f_\xi}|_{\mathbb{B}_\xi(\varepsilon)} \sqcup f \Big ) :\quad  \bigcup_{\xi \in \Xi_i} \mathbb{B}_\xi(\varepsilon) \cup X' \longrightarrow Y.$$
By Lemma~\ref{qtree}, since the Lipschitz property can be tested one pair of points at a time, we have in fact $\mathrm{Lip}(F_j) \leq 1$.
By Theorem~A, the $F_j$ admit $1$-Lipschitz extensions $\widehat{F}_j$ to $X$.
% Finally we define 
% $$F:=\sum_{j=1}^N \frac{1}{N} \widehat{F}_j \quad : X \longrightarrow Y.$$
% We claim that
% \begin{equation} \label{finalC'} \mathrm{Lip}(F) \leq 1 - \frac{1-\sqrt{C^*}}{N} =:C' <1. \end{equation}
Finally we claim that 
\begin{equation} \label{finalC'} F:=\sum_{j=1}^N \frac{1}{N} \widehat{F}_j \: : X \longrightarrow Y
~\text{ satisfies }~
 \mathrm{Lip}(F) \leq 1 - \frac{1-\sqrt{C^*}}{N} =:C' <1. \end{equation}
Indeed, this can be verified in restriction to each ball $\overline{\mathbb{B}_{\xi_i}(\varepsilon)}$ of the covering of~$X$: if $\xi_i$ falls in the bin $\Xi_j$, then on that ball $\widehat{F}_j$ is $\sqrt{C^*}$-Lipschitz by construction while all other $\widehat{F}_{j'}$ are $1$-Lipschitz; we conclude using~\eqref{moyenne}. \qed

\section{Conclusion} \label{seconj}
It seems natural to expect that the lower bound $K$ on curvature, and the upper bound $m$ on dimension, are not necessary in Theorem~\ref{main}.
\begin{conjecture}\label{conj}
For any $C<1$ there exists $C'\in (C,1)$ such that for any Hadamard manifolds $X,Y$ satisfying $\kappa_X\geq -1 \geq \kappa_Y$, every $C$-Lipschitz map from a subset of $X$ to $Y$ has a $C'$-Lipschitz extension to $X$: 
$$\mathcal{L}_{X,Y}(C)\leq C' < 1.$$
\end{conjecture}
\noindent This statement should still hold if both the map and its extension are required to be equivariant under a given pair of actions on $X$ and $Y$: see~\cite{guk}.% by a topological (e.g.\ discrete) group, the action on $X$ being proper.

\smallskip

Loss does occur, i.e.\ $C'>C$ in general, as testified by many examples. 
For instance, since $\ell \mapsto \mathcal{D}_{2\pi/3} (\ell, \ell)$ is strictly convex (see~\eqref{trigo}), a map $f$ that takes just the vertices of, say, a medium-sized equilateral triangle of $\mathbb{H}^2$ to the vertices of a smaller one, cannot be extended 
%to $\mathbb{H}^2$ 
without loss to the center of the triangle.
In such examples however, the ratio 
%$\frac{1-C'}{1-C}$ 
${(1-C')}/{(1-C)}$ 
never seems to get very small.
Thus we propose the following strengthening:
\begin{conjecture} \label{conj2}
There exists a universal $\alpha \in (0,1)$ such that $\mathcal{L}_{X,Y}(C) \leq C^\alpha$.% satisfies Conjecture~\ref{conj}.
\end{conjecture}
Interestingly, this conjecture appears to be open even for $C$ close to $0$.
The article~\cite{lps} shows that $\mathcal{L}_{X,Y}(C)/C$ is bounded above (which for small $C$ is a stronger property), but only under some extra assumptions on the Hadamard manifold~$Y$, such as fixed dimension with pinched curvature.

As $C$ approaches $1$, bounds on the constant $C'$ extracted from our proof of Theorem~\ref{main} are not very stringent. 
Fixing $K \leq -1$ and the dimension, we can estimate~\eqref{arcsinh} for $r=\frac{2\Delta}{1-C}$ to find that $1-C^*$ is on the order of $(1-C)^2$, yielding
$\varepsilon_0 \approx 1-C$, $\varepsilon \approx (1-C)^3$ and crucially $R\approx (1-C)^{-2}$ in~\eqref{buffer}.
In~\eqref{bins} this entails 
$N\approx \mathrm{e}^{-(\Lambda+o(1))/(1-C)^2}$ 
%$N \approx \mathrm{e}^{-\frac{\Lambda+o(1)}{(1-C)^2}}$
for some $\Lambda>0$, 
hence in~\eqref{finalC'} $$1-C'\approx \mathrm{e}^{-\frac{\Lambda+o(1)}{(1-C)^2}} \quad \text{as } C \rightarrow 1^-,$$
i.e.\ our upper bound $C'$ for $\mathcal{L}_{X,Y}(C)$ is a far cry from Conjecture~\ref{conj2}.

\end{document}